\documentclass[reqno,10pt,centertags]{amsart}
\usepackage[letterpaper,margin=1.5in,bottom=1.3in]{geometry}
\usepackage{amsmath,amsthm,amscd,amssymb,latexsym,enumerate}
\usepackage{comment}

\usepackage{hyperref}

\newcommand*{\mailto}[1]{\href{mailto:#1}{\nolinkurl{#1}}}
\newcommand{\arxiv}[1]{\href{http://arxiv.org/abs/#1}{arXiv: #1}}

\makeatletter
\def\theequation{\@arabic\c@equation}

\newcommand{\bbN}{{\mathbb{N}}}
\newcommand{\bbR}{{\mathbb{R}}}

\newcommand{\bbS}{{\mathbb{S}}}

\renewcommand{\a}{\alpha}

\newcommand{\g}{\gamma}

\newcommand{\no}{\nonumber}
\newcommand{\lb}{\label}
\newcommand{\bi}{\bibitem}

\newcommand{\ol}{\overline}
\newcommand{\bs}{\backslash}

\newcommand{\dott}{\,\cdot\,}

\newcommand{\supp}{\operatorname{supp}}

\renewcommand{\ln}{\operatorname{ln}}

\allowdisplaybreaks

\numberwithin{equation}{section}

\newtheorem{theorem}{Theorem}[section]
\newtheorem{lemma}[theorem]{Lemma}

\theoremstyle{definition}

\newtheorem{remark}[theorem]{Remark}

\begin{document}

\title[Log Refinements of a Power Weighted Hardy--Rellich-Type Inequality]{Logarithmic Refinements of a Power Weighted Hardy--Rellich-Type Inequality}

\author[F.\ Gesztesy]{Fritz Gesztesy}
\address{Department of Mathematics, 
Baylor University, Sid Richardson Bldg., 1410 S.~4th Street, Waco, TX 76706, USA}
\email{\mailto{Fritz\_Gesztesy@baylor.edu}}
\urladdr{\url{http://www.baylor.edu/math/index.php?id=935340}}

\author[M.\ M.\ H.\ Pang]{Michael M.\ H.\ Pang}
\address{Department of Mathematics,
University of Missouri, Columbia, MO 65211, USA}
\email{\mailto{pangm@missouri.edu}}
\urladdr{\url{https://www.math.missouri.edu/people/pang}}

\author[J.\ Stanfill]{Jonathan Stanfill}
\address{Department of Mathematics, The Ohio State University \\
100 Math Tower, 231 West 18th Avenue, Columbus, OH 43210, USA}
\email{\mailto{stanfill.13@osu.edu}}
\urladdr{\url{https://u.osu.edu/stanfill-13/}}

\dedicatory{Dedicated, with admiration, to Fedor Sukochev, mathematician extraordinaire} 
\date{\today}
\thanks{To appear in Ann. Funct. Anal.}
\@namedef{subjclassname@2020}{\textup{2020} Mathematics Subject Classification}

\begin{abstract}
The principal purpose of this note is to prove a logarithmic refinement of the power weighted Hardy--Rellich inequality on $n$-dimensional balls, valid for the largest variety of underlying parameters and for all dimensions $n \in \bbN$, $n\geq 2$.  
\end{abstract}

\maketitle


\section{Introduction} \lb{s1}

In the recent paper \cite{GPPS24} we reconsidered the following sharp inequality, first derived by Caldiroli and Musina \cite[Theorem~3.1]{CM12}, 
\begin{align}\lb{1.1}
\begin{split}
\int_{\bbR^n} |x|^\g|(\Delta f)(x)|^2 \, d^n x 
 \geq C_{n,\g} \int_{\bbR^n} |x|^{\g-4} |f(x)|^2 \, d^n x,& \\
 \g\in\bbR, \; f \in C^{\infty}_0(\bbR^n \backslash \{0\}), \; n\in\bbN, \, n\geq2,&
 \end{split} 
\end{align}
where
\begin{equation}\lb{1.2}
C_{n,\g}=\min_{j\in\bbN_0}\bigg\{\bigg(\frac{(n-2)^2}{4}-\frac{(\g-2)^2}{4}+j(j+n-2)\bigg)^2\bigg\}.
\end{equation}
In addition, we also derived the sharp inequality (sometimes called the Hardy--Rellich inequality), 
\begin{align}\lb{1.3}
\begin{split} 
\int_{\bbR^n} |x|^\gamma|(\Delta f)(x)|^2\, d^nx\geq A_{n,\gamma}\int_{\bbR^n}|x|^{\gamma-2}|(\nabla f)(x)|^2\, d^nx,&  \\
\gamma \in \bbR, \; f\in C_0^\infty(\bbR^n\backslash \{0\}),\; n\in\bbN,\, n\geq 2,&
\end{split} 
\end{align}
where
\begin{equation}\lb{1.4}
A_{n,\gamma}={\min}_{j \in \bbN_0}\{\alpha_{n,\gamma,\lambda_j}\},
\end{equation}
with
\begin{align}
\begin{split}
\alpha_{n,\gamma,\lambda_0}&=\alpha_{n,\gamma,0}=4^{-1}(n-\gamma)^2,     \lb{1.5} \\
\alpha_{n,\gamma,\lambda_j}&=\big[4^{-1}(n+\gamma-4)(n-\gamma)+\lambda_j\big]^2\Big/\big[4^{-1}(n+\gamma-4)^2+\lambda_j\big],\quad j\in\bbN. 
\end{split}
\end{align}

In the unweighted case $\gamma = 0$ this simplifies to the known fact,
\begin{equation}
A_{n,0} = \begin{cases}  n^2/4, & n \geq 5, \\
3, & n=4, \\
25/36, & n=3, \\
0, & n=2.
\end{cases}      \lb{1.6}
\end{equation}

In the special case where $C_{n,\gamma}$ in \eqref{1.1}, or $A_{n,\gamma}$ in \eqref{1.3}, vanishes, the resulting inequality is rendered trivial (e.g., there is no nontrivial inequality of the type \eqref{1.3} in the case $n=2$, $\gamma = 0$) and hence one wonders about the possibility of logarithmically refining these inequalities to prevent them from becoming insignificant.  

In this connection we recall that logarithmic refinements of \eqref{1.1} were already known. Indeed, as discussed in \cite{CM12}, whenever, $4^{-1} \big[(\gamma - 2)^2 - (n-2)^2\big]$ equals one of the eigenvalues of $- \Delta_{\bbS^{n-1}}$ (i.e., one of the numbers $j(j+n-2)$, $j \in \bbN_0$), then $C_{n,\gamma}$ vanishes, rendering inequality \eqref{1.1} trivial. In this context we recall the following result from \cite[Theorem~1.3]{GMP24}: 
\begin{align}
\begin{split} 
&\int_{B_n(0;R)} |x|^{\gamma} |(-\Delta f)(x)|^{2} \, d^{n}x \geq C_{n,\gamma} \int_{B_n(0;R)}  
|x|^{\gamma-4} | f(x)|^{2} \, d^{n}x    \\
&\quad + \big\{\big[ (n-\gamma)^{2} + (n+\gamma-4)^{2} \big]\big/16\big\}    \lb{1.7} \\
&\qquad \times \int_{B_n(0;R)} |x|^{\gamma-4}  \Bigg(\sum_{k=1}^{N} \prod_{p=1}^{k} [\ln_{p}(\eta/|x|)]^{-2} \Bigg)|f(x)|^{2} \, d^{n}x,  \\
& \, R \in (0,\infty), \; \gamma \in \bbR, \; N \in \bbN, \; \eta \in [e_{N}R,\infty), \; 
f \in C_{0}^{\infty}(B_n(0;R)\bs\{0\}), 
\end{split} 
\end{align}
which yields an appropriate logarithmic refinement even if $C_{n,\gamma}$ vanishes. Here $B_n(0;R)$ denotes the open ball in $\bbR^n$, $n \in \bbN$, $n \geq 2$, centered at the origin $0$ of radius $R \in (0,\infty)$, the iterated logarithms $\ln_k( \, \cdot \, )$, $k \in \bbN$, are given by
\begin{equation} 
\ln_1(\, \cdot \,) = \ln( \, \cdot \, ), \quad \ln_{k+1}( \, \cdot \,) = \ln \big( \ln_k(\, \cdot \,) \big), \quad k \in \bbN,  
\lb{1.8} 
\end{equation} 
and the iterated exponentials $e_j$, $j \in \bbN_0$, are introduced via 
\begin{equation}
e_0 = 0, \quad e_{j+1} = e^{e_j}, \quad j \in \bbN_{0}.      \lb{1.9} 
\end{equation} 

Given the logarithmic refinement \eqref{1.7} of \eqref{1.1}, it is natural to ask if a corresponding analogous  logarithmic refinement of \eqref{1.3} exists that prevents it from becoming insignificant if $A_{n,\gamma}$ vanishes. Answering this question in the affirmative is the principal purpose of this note. In particular, we will prove the following inequality in Theorem \ref{t2.3}:
\begin{align} 
\begin{split}
& \int_{B_n(0;R)} |x|^\g |(-\Delta f)(x)|^2\, d^n x \geq A_{n,\g} \int_{B_n(0;R)} |x|^{\g-2} |(\nabla f)(x)|^2\, d^n x  \\
& \quad + 4^{-1}\int_{B_n(0;R)} |x|^{\g-2} \Bigg(\sum_{k=1}^{N} \prod_{p=1}^{k} [\ln_{p}(\eta/|x|)]^{-2} \Bigg)|(\nabla f)(x)|^2\, d^n x    \lb{1.10} \\
& \quad + 4^{-1}\int_{B_n(0;R)}|x|^{\gamma-4}\Bigg(\sum_{k=1}^{N} \prod_{p=1}^{k} [\ln_{p}(\eta/|x|)]^{-2} \Bigg)|(\nabla_{\bbS^{n-1}} f)(x)|^2\, d^n x,   \\
& \, R \in (0,\infty), \; \gamma \in \bbR, \; N, n \in \bbN, \, n\geq2, \; \eta \in [e_{N}R,\infty), \; 
f \in C_{0}^{\infty}(B_n(0;R)\bs\{0\}).
\end{split} 
\end{align}
Once again, this inequality remains meaningful even if $A_{n,\g}$ vanishes. 

Given the enormity of the literature on (power weighted) Rellich and Hardy--Rellich-type inequalities, we will not repeat the extensive list (still necessarily incomplete) of references cited in \cite{GPPS24}, and so refer the reader to the latter. However, more specifically, we mention that Caldiroli and Musina \cite{CM12} proved in 2012 that the constant $C_{n,\gamma}$ in \eqref{1.1} is optimal. (For various restricted ranges of $\gamma$ see also Adimurthi, Grossi, and Santra \cite{AGS06}, Ghoussoub and Moradifam \cite{GM11}, \cite[Sects.~6.3, 6.5, Ch.~7]{GM13}, and Tertikas and Zographopoulos \cite{TZ07}.) The special unweighted case $\gamma =0$ was settled for $n \geq 5$ by Herbst \cite{He77} in 1977 and subsequently by Yafaev \cite{Ya99} in 1999 for $n \geq 3$, $n \neq 4$ (both authors consider much more general fractional inequalities).

Under various restrictions on $\gamma$, Tertikas and Zographopoulos \cite{TZ07} obtained in 2007 optimality of $A_{n,\gamma}$ for $n \geq 5$ and $\bbR^n$ replaced by appropriate open bounded domains $\Omega$ with $0 \in \Omega$. This is revisited in Ghoussoub and Moradifam \cite{GM11}, \cite[Part~2]{GM13}. 
Similarly, Tertikas and Zographopoulos \cite{TZ07} obtained optimality of $A_{n,0}$ for $n \geq 5$; Beckner \cite{Be08a} (see also \cite{Be08}), and subsequently, Ghoussoub and Moradifam \cite{GM11}, \cite[Sects.~6.3, 6.5, Ch.~7]{GM13} and Cazacu \cite{Ca20}, obtained optimality of $A_{n,0}$ for $n \geq 3$. 

As a notational comment we remark that we abbreviate $\bbN_0 = \bbN \cup \{0\}$, and denote by ${\bbS^{n-1}}$ the  unit sphere in $\bbR^n$, $n \in \bbN$, $n \geq 2$.

\section{A logarithmically Modifed Hardy--Rellich-type Inequality} \lb{s2}

We begin by recalling the following simplified version of \cite[Theorem 3.1 $(iii)$]{GLMP22}.

\begin{lemma}\lb{l2.1}
Let $R\in(0,\infty)$, $\a\in\bbR$, $N\in\bbN$, $\eta\in[e_N R,\infty)$, and $f\in C_0^\infty ((0,R))$. Then
\begin{align}
\begin{split}\lb{2.1a}
\int_0^R r^\a |f'(r)|^2\, dr&\geq 4^{-1}(1-\a)^2 \int_0^R r^{\a-2} |f(r)|^2\, dr\\
&\quad\, +4^{-1}\int_0^R r^{\a-2} \Bigg(\sum_{k=1}^{N} \prod_{p=1}^{k} [\ln_{p}(\eta/r)]^{-2} \Bigg)|f(r)|^2\, dr,
\end{split}
\end{align}
where the iterated logarithms $\ln_k( \, \cdot \, )$, $k \in \bbN$, are given by
\begin{equation} 
\ln_1(\, \cdot \,) = \ln( \, \cdot \, ), \quad \ln_{k+1}( \, \cdot \,) = \ln \big( \ln_k(\, \cdot \,) \big), \quad k \in \bbN, 
\end{equation} 
and the iterated exponentials $e_j$, $j \in \bbN_0$, are introduced via 
\begin{equation}
e_0 = 0, \quad e_{j+1} = e^{e_j}, \quad j \in \bbN_{0}.
\end{equation} 
\end{lemma}
\begin{proof}
As the current investigation came about while studying factorizations in \cite{GPPS24}, we provide a factorization proof of this lemma in the spirit of \cite{GPPS24} (see also \cite{GLMP18} for related higher dimensional unweighted factorizations with log refinements).

Given $R\in(0,\infty)$, $\a\in\bbR$, $N\in\bbN$, $\eta\in[e_N R,\infty)$, one defines the differential expression
\begin{equation}
T_{N,\a}=r^{\a/2}\frac{d}{dr}+\frac{\a-1}{2} r^{(\a-2)/2}+\frac{1}{2}r^{(\a-2)/2}\sum_{k=1}^{N} \prod_{p=1}^{k} [\ln_{p}(\eta/r)]^{-1},\quad r\in(0,R).
\end{equation}
Then, after applying appropriate integration by parts and combining similar terms, one confirms that for $f\in C_0^\infty((0,R))$,
\begin{align}
\begin{split} 
0\leq\int_0^R |(T_{N,\a}f)(r)|^2\, dr&= \int_0^R r^\a |f'(r)|^2\, dr-4^{-1}(1-\a)^2 \int_0^R r^{\a-2} |f(r)|^2\, dr\\
&\quad\, -4^{-1}\int_0^R r^{\a-2} \Bigg(\sum_{k=1}^{N} \prod_{p=1}^{k} [\ln_{p}(\eta/r)]^{-2} \Bigg)|f(r)|^2\, dr,
\end{split} 
\end{align}
proving \eqref{2.1a}.
\end{proof}

Before deriving our next result, we recall some standard notation and facts. 
Let $\mathbb{S}^{n-1}$ denote the $(n-1)$-dimensional unit sphere in $\bbR^n,\ n\in\bbN,\ n\geq 2,$ with $d^{n-1}\omega:=d^{n-1}\omega(\theta)$ the usual volume measure on $\mathbb{S}^{n-1}$. We denote by $-\Delta_{\mathbb{S}^{n-1}}$ the nonnegative, self-adjoint Laplace--Beltrami operator in $L^2(\mathbb{S}^{n-1};d^{n-1}\omega)$. Let
\begin{equation}
\lambda_j=j(j+n-2),\quad j\in\bbN_0,
\end{equation}
be the eigenvalues of $-\Delta_{\bbS^{n-1}}$, that is, $\sigma (-\Delta_{\bbS^{n-1}}) = \{j(j+n-2)\}_{j \in \bbN_0}$, of multiplicity
\begin{equation}
m(\lambda_j)=(2j+n-2)(j+n-2)^{-1} {j+n-2 \choose n-2},\quad j\in\bbN_0,
\end{equation}
with corresponding eigenfunctions $\varphi_{j,\ell},\ j\in\bbN_0,\ \ell\in\{1,\dots,m(\lambda_j)\}$. We may (and will) assume that $\{\varphi_{j,\ell}\}_{j\in\bbN_0,\ \ell\in\{1,\dots,m(\lambda_j)\}}$ is an orthonormal basis of $L^2({\bbS^{n-1}};d^{n-1}\omega)$, and let
\begin{align}
\begin{split}
F_{f,j,\ell}(r)=(\varphi_{j,\ell},f(r,\dott))_{L^2({\bbS^{n-1}};d^{n-1}\omega)}=\int_{\bbS^{n-1}} \overline{\varphi_{j,\ell}(\theta)}  f(r,\theta) \, d^{n-1}\omega(\theta),&\\
f\in C_0^\infty(\bbR^n\backslash\{0\}),\quad r>0,\quad j\in\bbN_0,\ \ell\in\{1,\dots,m(\lambda_j)\}.&
\end{split}
\end{align}
Finally, let $B_n(0;R)$ denote the open ball in $\bbR^n$, $n \in \bbN$, $n \geq 2$, centered at the origin $0$ of radius $R \in (0,\infty)$. 

We are now in the position to prove the following lemma which will be combined with Lemma \ref{l2.1} to prove our main result.

\begin{lemma}\lb{l2.2}
Let $R\in(0,\infty)$, $f\in C_0^\infty(B_n(0;R)\backslash\{0\})$, and $g\in C((0,R))$ satisfy $g(r)>0$ for all $r\in(0,R)$. Then
\begin{equation}\lb{2.4a}
\int_{B_n(0;R)} g(|x|)|(\nabla f)(x)|^2\, d^n x=\sum_{j=0}^\infty \sum_{\ell=1}^{m(\lambda_j)}\int_0^R g(r)\big[|F'_{f,j,\ell}(r)|^2 r^{n-1}+\lambda_j|F_{f,j,\ell}(r)|^2 r^{n-3}\big]\, dr.
\end{equation}
\end{lemma}
\begin{proof}
We begin by recalling that
\begin{equation}\lb{2.5a}
|(\nabla f)(x)|^2=|(\partial f/\partial r)(r,\theta)|^2+r^{-2}|(\nabla_{\bbS^{n-1}} f(r,\dott))(\theta)|^2, 
\end{equation}
where $\nabla_{\bbS^{n-1}}$ denotes the gradient operator on ${\bbS^{n-1}}$.
Thus applying \eqref{2.5a} and \cite[Lemma 2.1]{GMP24} yields
\begin{align}
& \int_{B_n(0;R)} g(|x|)|(\nabla f)(x)|^2\, d^n x =\int_0^R g(r)\int_{\bbS^{n-1}} |(\nabla f)(r,\theta)|^2 \, d^{n-1}\omega(\theta) \, r^{n-1}dr  \no \\
&\quad =\int_0^R g(r)\int_{\bbS^{n-1}} \big[|(\partial f/\partial r)(r,\theta)|^2 \no \\
&\quad \hspace{2.7cm} +r^{-2}|(\nabla_{\bbS^{n-1}} f(r,\dott))(\theta)|^2\big] \, d^{n-1}\omega(\theta) \, r^{n-1}dr  \no \\
&\quad =\int_0^R g(r)\Bigg\{ \int_{\bbS^{n-1}} \Bigg|\sum_{j=0}^\infty \sum_{\ell=1}^{m(\lambda_j)} F'_{f,j,\ell}(r)\varphi_{j,\ell}(\theta)\Bigg|^2 \, d^{n-1}\omega(\theta)  \no \\
& \quad \hspace{1.8cm} + r^{-2}\int_{\bbS^{n-1}} \ol{(-\Delta_{\bbS^{n-1}} f)(r,\theta)} f(r,\theta) \, d^{n-1}\omega(\theta)\Bigg\}\, r^{n-1} dr  \no \\
&\quad =\int_0^R g(r)\Bigg\{ \sum_{j=0}^\infty \sum_{\ell=1}^{m(\lambda_j)} |F'_{f,j,\ell}(r)|^2     \lb{2.11b} \\
&\qquad\, + r^{-2}\int_{\bbS^{n-1}}\Bigg[\sum_{j=0}^\infty \sum_{\ell=1}^{m(\lambda_j)}\lambda_j \ol{F_{f,j,\ell}(r)\varphi_{j,\ell}(\theta)}\Bigg] \no\\
&\quad \hspace{2.1cm}\times\Bigg[\sum_{j=0}^\infty \sum_{\ell=1}^{m(\lambda_j)} F_{f,j,\ell}(r)\varphi_{j,\ell}(\theta)\Bigg] \, d^{n-1}\omega(\theta)\Bigg\}\, r^{n-1} dr  \no \\
&\quad = \sum_{j=0}^\infty \sum_{\ell=1}^{m(\lambda_j)}\int_0^R g(r)\big[|F'_{f,j,\ell}(r)|^2 +\lambda_j r^{-2} |F_{f,j,\ell}(r)|^2\big]\, r^{n-1} dr,  \no
\end{align}
proving \eqref{2.4a}
\end{proof}

Explicitly, \eqref{2.11b} yields
\begin{align}
\begin{split} 
& \int_0^R g(r)\int_{\bbS^{n-1}} |(\partial f/\partial r)(r,\theta)|^2 \, d^{n-1}\omega(\theta) \, r^{n-1}dr   \lb{2.12b} \\
& \quad =
\sum_{j=0}^\infty \sum_{\ell=1}^{m(\lambda_j)}\int_0^R g(r) |F'_{f,j,\ell}(r)|^2 \, r^{n-1} dr,   
\end{split} \\
\begin{split} 
& \int_0^R g(r)\int_{\bbS^{n-1}} r^{-2}|(\nabla_{\bbS^{n-1}} f(r,\dott))(\theta)|^2 \, d^{n-1}\omega(\theta) \, r^{n-1}dr  
\lb{2.13b} \\
& \quad = 
\sum_{j=0}^\infty \sum_{\ell=1}^{m(\lambda_j)} \lambda_j \int_0^R g(r) r^{-2} |F_{f,j,\ell}(r)|^2\, r^{n-1} dr. 
\end{split} 
\end{align}

The previous results now allow us to prove the main result in this note in the form of the following Hardy--Rellich-type inequality with logarithmic refinements.

\begin{theorem}\lb{t2.3}
Let $R\in(0,\infty)$, $\g\in\bbR$, $N, n\in\bbN$, with $n\geq2$, $\eta\in[e_N R,\infty)$, and $f\in C_0^\infty (B_n(0;R)\backslash\{0\})$. Then
\begin{align} 
& \int_{B_n(0;R)} |x|^\g |(-\Delta f)(x)|^2\, d^n x \geq A_{n,\g} \int_{B_n(0;R)} |x|^{\g-2} |(\nabla f)(x)|^2\, d^n x \no \\
& \quad + 4^{-1}\int_{B_n(0;R)} |x|^{\g-2} \Bigg(\sum_{k=1}^{N} \prod_{p=1}^{k} [\ln_{p}(\eta/|x|)]^{-2} \Bigg)|(\nabla f)(x)|^2\, d^n x    \lb{2.7a} \\
& \quad + 4^{-1}\int_{B_n(0;R)}|x|^{\gamma-4}\Bigg(\sum_{k=1}^{N} \prod_{p=1}^{k} [\ln_{p}(\eta/|x|)]^{-2} \Bigg)|(\nabla_{\bbS^{n-1}} f)(x)|^2\, d^n x,   \no 
\end{align}
where
\begin{equation}\lb{2.8a}
A_{n,\gamma}={\min}_{j \in \bbN_0}\{\alpha_{n,\gamma,\lambda_j}\},
\end{equation}
with
\begin{align}
\begin{split}
\alpha_{n,\gamma,\lambda_0}&=\alpha_{n,\gamma,0}=4^{-1}(n-\gamma)^2,     \lb{2.8b} \\
\alpha_{n,\gamma,\lambda_j}&=\big[4^{-1}(n+\gamma-4)(n-\gamma)+\lambda_j\big]^2\Big/\big[4^{-1}(n+\gamma-4)^2+\lambda_j\big],\quad j\in\bbN. 
\end{split}
\end{align} 
Excluding the cases $(\alpha)$ $n=2,\ \gamma=2$ and $(\beta)$ $n=3,\ \gamma=1$, the constant $A_{n,\gamma}$ on the right-hand side of inequality \eqref{2.7a} is optimal. 
\end{theorem}
\begin{proof}
By \cite[Eq. (A.25)]{GPPS24} and \cite[Lemmas 2.3 and B.3 $(i)$]{GMP24} one has
\begin{align}\lb{2.9a}
& \int_{B_n(0;R)} |x|^\g |(-\Delta f)(x)|^2\, d^n x    \no \\
& \quad =\sum_{j=0}^\infty \sum_{\ell=1}^{m(\lambda_j)}\int_0^R r^{\g+n-1}\big|-r^{1-n}\big[d/dr\big(r^{n-1} F'_{f,j,\ell}(r)\big)\big] +\lambda_j r^{-2} F_{f,j,\ell}(r)\big|^2\, dr \no\\
& \quad =\sum_{j=0}^\infty \sum_{\ell=1}^{m(\lambda_j)} \Bigg\{\int_0^R r^{\g+n-1}\big|F''_{f,j,\ell}(r)\big|^2\, dr \\
&\hspace{3cm}+[2\lambda_j+(n-1)(1-\g)]\int_0^R r^{\g+n-3}\big|F'_{f,j,\ell}(r)\big|^2\, dr \no\\
&\hspace{3cm}+\lambda_j[\lambda_j+(\g+n-4)(2-\g)]\int_0^R r^{\g+n-5} |F_{f,j,\ell}(r)|^2\, dr\Bigg\}.\no
\end{align}
Furthermore, note that \eqref{2.8a} implies
\begin{equation}
\lambda_j A_{n,\g}\leq \big[4^{-1}(n+\g-4)(n-\g)+\lambda_j\big]^2-4^{-1}(n+\g-4)^2 A_{n,\g}, \quad j \in \bbN_0, 
\end{equation}
or equivalently,
\begin{equation}\lb{2.11a}
\lambda_j A_{n,\g}\leq 4^{-1}(n+\g-4)^2\big[4^{-1}(n-\g)^2+2\lambda_j-A_{n,\g}\big]+\lambda_j[\lambda_j+(n+\g-4)(2-\g)], \quad j \in \bbN_0.
\end{equation}
Applying Lemma \ref{l2.1} and \eqref{2.11a} to \eqref{2.9a} yields
\begin{align} 
& \int_{B_n(0;R)} |x|^\g |(-\Delta f)(x)|^2\, d^n x \geq \sum_{j=0}^\infty \sum_{\ell=1}^{m(\lambda_j)} \Bigg\{4^{-1}(2-n-\g)^2\int_0^R r^{\g+n-3}\big|F'_{f,j,\ell}(r)\big|^2\, dr \no\\
&\qquad +4^{-1}\int_0^R r^{\g+n-3}\big|F'_{f,j,\ell}(r)\big|^2 \Bigg(\sum_{k=1}^{N} \prod_{p=1}^{k} [\ln_{p}(\eta/r)]^{-2} \Bigg)\, dr  \no\\
& \qquad +[2\lambda_j+(n-1)(1-\g)]\int_0^R r^{\g+n-3}\big|F'_{f,j,\ell}(r)\big|^2\, dr \no\\
&\qquad +\lambda_j[\lambda_j+(\g+n-4)(2-\g)]\int_0^R r^{\g+n-5} |F_{f,j,\ell}(r)|^2\, dr\Bigg\} \no \\
& \quad =\sum_{j=0}^\infty \sum_{\ell=1}^{m(\lambda_j)} \Bigg\{\big[4^{-1}(n-\g)^2+2\lambda_j\big]\int_0^R r^{\g+n-3}\big|F'_{f,j,\ell}(r)\big|^2\, dr \no\\
&\qquad +4^{-1}\int_0^R r^{\g+n-3}\big|F'_{f,j,\ell}(r)\big|^2 \Bigg(\sum_{k=1}^{N} \prod_{p=1}^{k} [\ln_{p}(\eta/r)]^{-2} \Bigg)\, dr  \no\\
&\qquad +\lambda_j[\lambda_j+(\g+n-4)(2-\g)]\int_0^R r^{\g+n-5} |F_{f,j,\ell}(r)|^2\, dr\Bigg\}  \no \\
& \quad =\sum_{j=0}^\infty \sum_{\ell=1}^{m(\lambda_j)} \Bigg\{A_{n,\g}\int_0^R r^{\g+n-3}\big|F'_{f,j,\ell}(r)\big|^2\, dr \no\\
&\qquad +4^{-1}\int_0^R r^{\g+n-3}\big|F'_{f,j,\ell}(r)\big|^2 \Bigg(\sum_{k=1}^{N} \prod_{p=1}^{k} [\ln_{p}(\eta/r)]^{-2} \Bigg)\, dr  \no\\
&\qquad +\big[4^{-1}(n-\g)^2+2\lambda_j-A_{n,\g}\big]\int_0^R r^{\g+n-3}\big|F'_{f,j,\ell}(r)\big|^2\, dr \no\\
&\qquad +\lambda_j[\lambda_j+(\g+n-4)(2-\g)]\int_0^R r^{\g+n-5} |F_{f,j,\ell}(r)|^2\, dr\Bigg\} \no \\
& \quad \geq\sum_{j=0}^\infty \sum_{\ell=1}^{m(\lambda_j)} \Bigg\{A_{n,\g}\int_0^R r^{\g+n-3}\big|F'_{f,j,\ell}(r)\big|^2\, dr \no\\
& \qquad +4^{-1}\int_0^R r^{\g+n-3}\big|F'_{f,j,\ell}(r)\big|^2 \Bigg(\sum_{k=1}^{N} \prod_{p=1}^{k} [\ln_{p}(\eta/r)]^{-2} \Bigg)\, dr  \no\\
& \qquad +\big\{4^{-1}(\g+n-4)^2\big[4^{-1}(n-\g)^2+2\lambda_j-A_{n,\g}\big] \no\\
&\qquad \quad +\lambda_j[\lambda_j+(\g+n-4)(2-\g)]\big\}\int_0^R r^{\g+n-5} |F_{f,j,\ell}(r)|^2\, dr \no \\
& \qquad +4^{-1}\big[4^{-1}(n-\g)^2+2\lambda_j-A_{n,\g}\big]\no \\
&\qquad \quad \times\int_0^R r^{\g+n-5}\big|F_{f,j,\ell}(r)\big|^2 \Bigg(\sum_{k=1}^{N} \prod_{p=1}^{k} [\ln_{p}(\eta/r)]^{-2} \Bigg)\, dr\Bigg\} \no\\
&\quad \geq \sum_{j=0}^\infty \sum_{\ell=1}^{m(\lambda_j)} \Bigg\{A_{n,\g}\int_0^R \big[r^{\g+n-3}\big|F'_{f,j,\ell}(r)\big|^2 + \lambda_j r^{\g+n-5}|F_{f,j,\ell}(r)|^2\big]\, dr \no\\
&\qquad +4^{-1}\int_0^R r^{\g-2}\Bigg(\sum_{k=1}^{N} \prod_{p=1}^{k} [\ln_{p}(\eta/r)]^{-2} \Bigg) 
 \big[\big|F'_{f,j,\ell}(r)\big|^2 r^{n-1}+\lambda_j |F_{f,j,\ell}(r)|^2 r^{n-3}\big]\, dr\Bigg\}    \no \\
& \qquad + \sum_{j=0}^\infty \sum_{\ell=1}^{m(\lambda_j)} 4^{-1}\lambda_j \int_0^R r^{\g+n-5}\big|F_{f,j,\ell}(r)\big|^2 \Bigg(\sum_{k=1}^{N} \prod_{p=1}^{k} [\ln_{p}(\eta/r)]^{-2} \Bigg)\, dr,   \lb{2.12a}
\end{align}
where we used the fact that $4^{-1}(n-\g)^2\geq A_{n,\g}$ (following from letting $j=0$ in \eqref{2.8a}) in the last inequality. Finally, applying Lemma \ref{l2.2} to the last inequality in \eqref{2.12a} with 
\begin{equation}
g(r)=r^{\g-2} \, \text{ and } \, g(r)=r^{\g-2} \Bigg(\sum_{k=1}^{N} \prod_{p=1}^{k} [\ln_{p}(\eta/r)]^{-2} \Bigg), \quad 
r \in (0,R), 
\end{equation}
one obtains, employing \eqref{2.12b} and \eqref{2.13b}, 
\begin{align}
& \int_{B_n(0;R)} |x|^\g |(-\Delta f)(x)|^2\, d^n x \geq A_{n,\g} \int_{B_n(0;R)} |x|^{\g-2} |(\nabla f)(x)|^2\, d^n x  \no \\
& \quad + 4^{-1}\int_{B_n(0;R)} |x|^{\g-2} \Bigg(\sum_{k=1}^{N} \prod_{p=1}^{k} [\ln_{p}(\eta/|x|)]^{-2} \Bigg)|(\nabla f)(x)|^2\, d^n x    \\
& \quad + 4^{-1}\int_{B_n(0;R)}|x|^{\gamma-4}\Bigg(\sum_{k=1}^{N} \prod_{p=1}^{k} [\ln_{p}(\eta/|x|)]^{-2} \Bigg)|(\nabla_{\bbS^{n-1}} f)(x)|^2\, d^n x.     \no
\end{align}

To prove optimality of $A_{n,\g}$ (excluding the cases $(\alpha)$ $n=2$, $\gamma=2$ and $(\beta)$ $n=3$,  $\gamma=1$), one can modify the proof of optimality found in \cite[Theorem~A.7]{GPPS24}, and we now recall the major steps of the latter. That proof begins by choosing a sequence $\{f_m\}_{m\in\bbN} \subset C_0^\infty((0,\infty))$ such that $f_m$ is real-valued and $f_m\not\equiv0$ for all $m\in \bbN$, and
\begin{equation}
\lim_{m\to\infty}\left(\int_0^\infty r^{\gamma+n-1}|f_m^{\prime \prime}(r)|^2\, dr\right)\left(\int_0^\infty r^{\gamma+n-5}|f_m(r)|^2\, dr\right)^{-1}=\frac{(2-\gamma-n)^2(4-\gamma-n)^2}{16}. 
\end{equation}
Next, depending on the values of $\gamma$ and $n$, one chooses an eigenfunction, $\varphi$, of $-\Delta_{\bbS^{n-1}}$ and defines $g_m\in C_0^\infty(\bbR^n\backslash\{0\}),\ m\in\bbN,$ by
\begin{equation}
g_m(x)=g_m(r,\theta)=f_m(r)\varphi(\theta),\quad x\in\bbR^n\backslash\{0\}.
\end{equation}
One can then show that
\begin{equation}
\dfrac{\int_{\bbR^n} |x|^\gamma|(-\Delta g_m)(x)|^2\, d^n x}{\int_{\bbR^n} |x|^{\gamma-2}|(\nabla g_m)(x)|^2\, d^n x}\underset{m\to\infty}{\longrightarrow}A_{n,\gamma},
\end{equation}
completing the proof of optimality in \cite[Theorem~A.7]{GPPS24}. To modify this proof for the current purpose of proving optimality of $A_{n,\g}$ in \eqref{2.7a}, one needs to choose a new sequence in $C_0^\infty((0,R))$ rather than $C_0^\infty((0,\infty))$ to begin with. To this end, we choose $\{f_m\}_{m\in\bbN} \subset C_0^\infty((0,\infty))$ as above and let $f_m\in C_0^\infty((0,\rho_m))$ (e.g., $\rho_m \geq  [\sup (\supp(f_m))] + 1$) for all $m\in\bbN$. We then define, for all $m\in\bbN$, $\widehat{f}_m\in C_0^\infty((0,R))$ by
\begin{equation}
\widehat{f}_m(y)=f_m(\rho_m y/R),\quad 0<y<R.
\end{equation}
One then readily verifies that
\begin{equation}
\lim_{m\to\infty}\left(\int_0^\infty r^{\gamma+n-1}|\widehat{f}_m^{\, \prime \prime}(r)|^2\, dr\right)\left(\int_0^\infty r^{\gamma+n-5}|\widehat{f}_m(r)|^2\, dr\right)^{-1}=\frac{(2-\gamma-n)^2(4-\gamma-n)^2}{16}.
\end{equation}
Thus, replacing $\{f_m\}_{m\in\bbN} \subset C_0^\infty((0,\infty))$ by $\{\widehat{f}_m\}_{m\in\bbN}\in C_0^\infty((0,R))$ in the proof of \cite[Theorem~A.7]{GPPS24} shows optimality of $A_{n,\g}$ in \eqref{2.7a}, once again, excluding the cases $(\alpha)$ $n=2,\ \gamma=2$ and 
$(\beta)$ $n=3,\ \gamma=1$.
\end{proof}

\begin{remark} \lb{r2.4}
$(i)$ The proof of Theorem \ref{t2.3} is similar to proofs found in \cite[Chs.~6, 7]{GM13}, but due to our application of \cite[Theorem 3.1 $(iii)$]{GLMP22} in Lemma \ref{l2.1}, the range of parameters has now been greatly extended in Theorem \ref{t2.3}, in particular, the two-dimensional case $n=2$ in inequality \eqref{2.7a} appears to have no precedent. \\[1mm] 
$(ii)$ In \cite[Theorems~A.5 and A.7]{GPPS24} the authors proved Theorem \ref{t2.3} without the log refinement terms (i.e., without the last two terms on the right side of (2.14)) and for a larger function space $C^{\infty}_0(\mathbb{R}^n \backslash \{0\})$. But even with this larger function space and without the log refinement terms, due to the method of proof, the authors were unable to show optimality of $A_{n, \gamma}$ in the two excluded cases in Theorem \ref{t2.3}, that is, for $(\alpha)$ $n = 2$ and $\gamma = 2$, and $(\beta)$ $n = 3$ and $\gamma = 1$. So, the optimality of $A_{n, \gamma}$ for those two cases remains open. 
\\[1mm] 
$(iii)$ We note that the inequality \eqref{2.7a} was formulated for the smallest natural function space $f \in C_{0}^{\infty}(B_n(0;R)\bs\{0\})$. Thus, at least  in principle, the optimal constants could have increased in the process when compared to the function spaces $f \in C_{0}^{\infty}(B_n(0;R))$ typically employed in \cite{AGS06}, \cite{Be08a}, \cite{Ca20}, \cite{GM11}, \cite[Part~2]{GM13}, \cite{TZ07}, etc. Interestingly enough, Theorem \ref{t2.3} demonstrates this possible increase in optimal constants is not happening with $A_{n,\gamma}$. In this context we note that \cite[Ch.~6]{GM13} derive optimality of $A_{n,\gamma}$ for $f \in C_0^\infty (B_n(0;R))$. \\[1mm] 
$(iv)$ Of course, by restriction, the principal inequalities in this paper (such as \eqref{2.7a}--\eqref{2.8b}) extend to the case where $f \in C_0^{\infty} (B_n(0;R) \backslash \{0\})$, $n \in \bbN$, $n \geq 2$, is replaced by $f \in C_0^{\infty} (\Omega \backslash \{0\})$, where $\Omega \subseteq B_n(0;R)$ is open and bounded with $0 \in \Omega$, without changing the constants in these inequalities. 
\hfill $\diamond$
\end{remark}

\medskip



\end{document}